\documentclass[12pt,english]{amsart}
\usepackage[cp1251]{inputenc}
\usepackage[english]{babel}
\usepackage{amsmath,amsthm,amssymb}
\newtheorem{theorem}{Theorem}[section]

\newtheorem{corollary}[theorem]{Corollary}

\newtheorem{proposition}[theorem]{Proposition}
\theoremstyle{definition}
\newtheorem{definition}[theorem]{Definition}

\numberwithin{equation}{section}

\begin{document}

\title[]{Distributions of polynomials in Gaussian random variables
under structural constraints}

\author{Egor D. Kosov}

\begin{abstract}
We study the regularity of densities of distributions
that are polynomial images of the standard Gaussian measure on $\mathbb{R}^n$.
We assume that the degree of a polynomial is fixed
and that each variable enters to a power bounded by
another fixed number.
\end{abstract}

\maketitle



\section{Introduction}

In this note we study the regularity properties
of distributional densities of
random variables
\begin{equation}\label{rv}
f(X):=\sum_{j_1=0}^{m}\ldots\sum_{j_n=0}^{m}a_{j_1,\ldots, j_n}X_1^{j_1}\cdot\ldots\cdot X_n^{j_n},
\end{equation}
where $X_1,\ldots, X_n$ are i.i.d. $\mathcal{N}(0,1)$ random variables.
The regularity properties of general polynomial images  of measures of different classes
have been extensively studied in recent years
(see \cite{B16}, \cite{B19},  \cite{BKZ16}, \cite{BKZ}, \cite{Kos}, and \cite{Kos-FCAA}).
This research has been motivated by the paper \cite{NP} about
the connection between the total variation and Kantorovich--Rubinstein (Fortet--Mourier) distances
on the set of distributions of polynomials
(see also \cite{BC}, \cite{BC17}, \cite{BC19}, \cite{BCP}, \cite{BZ}, and~\cite{NNP}).
It turns out that the regularity of distributions plays a crucial role
in estimates between these two distances
(see \cite{BKZ}, \cite{Kos},  and \cite{Kos-FCAA}).
In particular, the following theorem was proved in~\cite{Kos}.
\vskip .1in

{\it
The distribution density $\varrho_g$ of any non-constant random variable of the form
\begin{equation}\label{rv2}
g(X):= \sum_{j_1+\ldots+j_n\le d}a_{j_1,\ldots, j_n}X_1^{j_1}\cdot\ldots\cdot X_n^{j_n},
\end{equation}
where $X:=(X_1,\ldots, X_n)$ is a Gaussian random vector,
belongs to the Nikolskii--Besov space~$B_{1,\infty}^{1/d}$.
Moreover, for any fixed $d\in \mathbb{N}$
there is a constant $C(d)$ such that
$$
\omega(\varrho_g, \varepsilon):=
\sup\limits_{|h|\le \varepsilon}\int_\mathbb{R}|\varrho_g(s+h) - \varrho_g(s)|\, ds
\le C(d)[\mathbb{D}g(X)]^{-1/2d} \varepsilon^{1/d}
$$
for any random variable $g(X)$ of the form (\ref{rv2}),
where $\mathbb{D}g(X) := \mathbb{E}\bigl(g(X) - \mathbb{E}g(X)\bigr)^2$.
}

\vskip .1in
We recall  (see \cite{BIN}, \cite{Stein}) that the Nikolskii--Besov space~$B_{1,\infty}^\alpha$, with $\alpha\in(0,1)$,
consists of all functions $\varrho\in L^1(\mathbb{R})$ such that there is a constant $C>0$
for which
$$
\int_{\mathbb{R}}|\varrho(t+h) - \varrho(t)|\, dt\le C\, |h|^\alpha\quad \forall h\in \mathbb{R}.
$$

In this note we study bounds for the modulus of continuity
$\omega(\varrho_f,\cdot)$ for polynomials $f$ of the form (\ref{rv}),
i.e., we consider polynomial images of the standard Gaussian measure on
$\mathbb{R}^n$ when the polynomial $f$ is of
degree $d$ and satisfies the following additional constraint on its structure:

{\it each variable enters \eqref{rv} only to a  power not greater than some fixed number $m\le d$.}

Our main result is as follows.

\begin{theorem}\label{T-main}
Let $m,d\in\mathbb{N}$, $d\ge m$,
let $X:=(X_1,\ldots, X_n)$ be the standard Gaussian $n$-dimensional
random vector with independent coordinates.
Then there is a constant $C(m, d)$ depending only on $d$ and $m$,
such that,
for any non-constant polynomial
$$
f(x):=\sum_{j_1=0}^{m}\ldots\sum_{j_n=0}^{m}a_{j_1,\ldots, j_m}x_1^{j_1}\cdot\ldots\cdot x_n^{j_n}
$$
with the degree $d[f]:=\max\{j_1+\ldots+j_n\colon a_{j_1,\ldots, j_m}\ne0\}\le d$,
one has
$$
\omega(\varrho_f, \varepsilon)\le
C(m,d) (\varepsilon/a[f])^{1/m}\bigl[|\ln(\varepsilon/a[f])|^{d-m} + 1\bigr]
$$
where $a[f]:=\max\limits_{j_1+\ldots+j_n=d[f]}|a_{j_1,\ldots, j_n}|$
and where $\varrho_f$ is the density of the random variable $f(X)$.
\end{theorem}

This paper has been partially motivated by the papers \cite{GPU} and \cite{Ul},
where some bounds for the characteristic functions of random variables of type (\ref{rv})
are obtained, including the following estimate (see \cite[Theorem 5]{GPU}):
$$
|\mathbb{E}\exp\{it f(X)\}| \le C(n,m,d[f])|a[f]\cdot t|^{-1/m}\bigl[\ln(2 + |a[f]t|)\bigr]^\alpha
$$
where $\alpha = \frac{1}{2}\bigl(3n - \frac{d[f]}{m}\bigr) - 1$, $d[f]:=\max\{j_1+\ldots+j_n\colon a_{j_1,\ldots, j_m}\ne0\}$,
$a[f]:=\max\limits_{j_1+\ldots+j_n=d[f]}|a_{j_1,\ldots, j_n}|$.
As a corollary of Theorem \ref{T-main} we deduce a somewhat
sharper bound for the characteristic function, independent of the number of variables $n$,
with $\alpha = d[f]-m$.

\section{Definitions, notation, and known results}

In this section we introduce the definitions and notation
used throughout  the paper.
We also formulate several known results which will be important
in the proof of the main result.

\begin{definition}\label{D1}
For a function $\varrho\in L^1(\mathbb{R})$ and $\varepsilon>0$, we set
$$
\omega(\varrho, \varepsilon):= \sup\limits_{|h|\le \varepsilon}\int_\mathbb{R}|\varrho(s+h)-\varrho(s)|\, ds.
$$
\end{definition}

This is the usual $L^1$ modulus of continuity.

We use the notation
$$
\|\varphi\|_\infty:= \sup_{x\in \mathbb{R}}|\varphi(x)|
$$
for a function $\varphi\in C_0^\infty(\mathbb{R})$,
where $C_0^\infty(\mathbb{R})$ is the class of all
infinitely differentiable compactly supported functions.

In \cite{Kos-Sb} (see also \cite{Kos-FCAA}),
the following modulus of continuity was introduced.

\begin{definition}\label{D2}
For a function $\varrho\in L^1(\mathbb{R})$ and $\varepsilon>0$, we set
$$
\sigma(\varrho, \varepsilon):=
\sup\Bigl\{\int\varphi'(s)\varrho(s)\, ds:\,
\|\varphi\|_\infty\le \varepsilon,\, \|\varphi'\|_\infty\le1\Bigr\},
$$
where the supremum is taken over all functions $\varphi\in C_0^\infty(\mathbb{R})$.
\end{definition}

We note that $\sigma(\varrho, \cdot)$
is a monotone concave modulus of continuity (see \cite[Lemma 2.1]{Kos-Sb}).

According to \cite[Theorem 2.1]{Kos-Sb}, one has the following
equivalence of these two moduli of continuity.

\begin{proposition}\label{prop2}
For any function $\varrho\in L^1(\mathbb{R})$, we have
$$
2^{-1}\omega(\varrho, 2\varepsilon)\le
\sigma(\varrho, \varepsilon)
\le 6\, \omega(\varrho, \varepsilon).
$$
\end{proposition}

We note that for the density $\varrho$
of a random variable $W$
the modulus of continuity
$\sigma(\varrho, \cdot)$ is calculated as follows:
$$
\sigma(\varrho, \varepsilon):=
\sup\Bigl\{\mathbb{E}\varphi'(W):\,
\|\varphi\|_\infty\le \varepsilon,\, \|\varphi'\|_\infty\le1\Bigr\},
$$
where the supremum is taken over all functions $\varphi\in C_0^\infty(\mathbb{R})$.

In \cite[Corollary 2.2]{Kos-Sb}, the following result is also obtained.

\begin{proposition}\label{prop1}
Assume that the distribution of some random variable $W$
is absolutely continuous and has density $\varrho$.
Then
$$
P(W\in A)\le \sigma\bigl(\varrho, \lambda(A)\bigr)
$$
for each Borel set $A\subset \mathbb{R}$.
Here $\lambda$ denotes the standard Lebesgue measure on $\mathbb{R}$.
\end{proposition}

We will need the following theorem from \cite{Kos} (see also \cite{BKZ} and \cite{Kos-FCAA}).

\begin{theorem}\label{t1}
Let $d\in\mathbb{N}$.
Then there is a constant $C(d)$ depending only on  $d$
such that, for any polynomial $g$ of degree at most $d$, any Gaussian random vector $X$,
and any function $\varphi\in C_0^\infty(\mathbb{R})$, one has
$$
[\mathbb{D}g(X)]^{1/2d}\mathbb{E}\, \varphi'\bigl(g(X)\bigr)\le
C(d)\|\varphi\|_\infty^{1/d}\|\varphi'\|_\infty^{1-1/d}
$$
where $\mathbb{D}g(X)$ is the variance of $g(X)$.
In particular,
$$
\sigma(\varrho_g, \varepsilon)\le C(d)[\mathbb{D}g(X)]^{-1/2d}\varepsilon^{1/d}.
$$
\end{theorem}

\section{Proof of Theorem \ref{T-main}}

We first note that by Proposition \ref{prop2} we can work with
$\sigma(\varrho_f,\cdot)$ in place of $\omega(\varrho_f,\cdot)$.
Since $\sigma(\varrho_{\alpha f},t) = \sigma(\varrho_f,t/\alpha)$,
it is sufficient to prove the bound for polynomials
$f$ with $a[f]=1$.

We now assume that $a[f]=1$.
In this case the proof will be done by induction on $d$.
We note that the number $a[f]$ is a coefficient of some monomial
$x_1^{j^0_1}\cdot\ldots\cdot x_n^{j^0_n}$.

Let us assume that $d=m$. Due to the equivalence
of the $L^2$-norm
and the Sobolev norm of the Gaussian Sobolev space $W^{2,d}$
on the space of all polynomials of degree not greater than~$d$
(see \cite[Corollary 5.5.5 and Theorem 5.7.2]{Gaus})
one has
$$
[\mathbb{D}f(X)]^{1/2}
\ge
c_1(d)
\Bigl[\mathbb{E}\Bigl|
\frac{\partial^{j_1^0+\ldots+j_n^0}f}{\partial x_1^{j_1^0}\ldots\partial x_n^{j_n^0}}(X)\Bigr|^2\Bigr]^{1/2}
=
c_1(d)\cdot j_1^0!\cdot\ldots\cdot j_n^0!\cdot a[f]\ge c_1(d).
$$
By Theorem \ref{t1} we have
$$
\sigma(\gamma_n\circ f^{-1}, t)\le C(d)[c_1(d)]^{-1/d}t^{1/d}.
$$
Thus, the base case $d=m$ of induction is proved.

We now make the inductive step.
Let $Z$ be the standard Gaussian random variable independent
of the random vector~$X$.
We note that for any $\varepsilon>0$
one has
$$
\mathbb{E}\, \varphi'\bigl(f(X)\bigr)
= \mathbb{E}\bigl[\varphi'\bigl(f(X)\bigr)-\varphi'\bigl(f(X) + \varepsilon Z\bigr)\bigr]
+
\mathbb{E}\, \varphi'\bigl(f(X) + \varepsilon Z\bigr).
$$
For the first term we have
\begin{multline*}
\mathbb{E}\bigl[\varphi'\bigl(f(X)\bigr)-\varphi'\bigl(f(X) + \varepsilon Z\bigr)\bigr]
\le
\|\varphi'\|_\infty\, \mathbb{E}_Z\Bigl[\int_\mathbb{R}
|\varrho_f(t) - \varrho_f(t - \varepsilon Z)|\,dt\Bigr]
\\
\le
2\|\varphi'\|_\infty\, \mathbb{E}_Z\bigl[\sigma(\varrho_f, \varepsilon |Z|)\bigr]
\le
2\|\varphi'\|_\infty\, \sigma\bigl(\varrho_f, \varepsilon\cdot \mathbb{E}|Z|\bigr)
\le 2\|\varphi'\|_\infty\, \sigma(\varrho_f, \varepsilon),
\end{multline*}
where we have used Proposition \ref{prop2}, the concavity and the monotonicity of the function
$\sigma(\varrho_f, \cdot)$.

We now estimate the second term.
As we have already mentioned, the number $a[f]$ is a coefficient of some monomial
$x_1^{j^0_1}\cdot\ldots\cdot x_n^{j^0_n}$.
Without loss of generality we can assume that $j^0_n\ne 0$.
Consider the polynomial $f$
as a polynomial of the $n$-th variable~$x_n$:
$$
f(x_1,\ldots, x_{n-1},x_n) = \sum_{j=0}^{m}f_j(x_1,\ldots, x_{n-1}) x_n^j.
$$
We apply Theorem \ref{t1} to the random variable
$f(x_1,\ldots, x_{n-1}, X_n, Z)$, which gives the bound
\begin{multline*}
\mathbb{E}\, \varphi'\bigl(f(X) + \varepsilon Z\bigr)
=
\mathbb{E}_{X_1,\ldots, X_{n-1}} \mathbb{E}_{X_n,Z}\,
\varphi'\bigl(f(X_1,\ldots, X_{n-1}, X_n) + \varepsilon Z\bigr)
\\
\le
C(m)\|\varphi\|_\infty^{1/m}\|\varphi'\|_\infty^{1-1/m}
\mathbb{E}_{X_1,\ldots, X_{n-1}}
\bigl(\mathbb{D}_{X_n}f(X_1,\ldots, X_{n-1},X_n)+\varepsilon^2\bigr)^{-1/2m}.
\end{multline*}
We recall that
for any polynomial $g(s)=\sum\limits_{j=0}^{m}a_js^j$,
by the Hermite polynomial expansion,
one has
$$\mathbb{D}g(X_n)\ge\frac{1}{m}\mathbb{E}|g'(X_n)|^2\ge c_2(m)\max\limits_{1\le j\le m}|a_j|^2,$$
where the last inequality
follows from the equivalence of any two norms on a finite-dimensional space.
Without loss of generality we assume that $c_2(m)\le1$.
Thus,
\begin{multline*}
\mathbb{E}\, \varphi'\bigl(f(X) + \varepsilon Z\bigr)
\le
C(m)\|\varphi\|_\infty^{1/m}\|\varphi'\|_\infty^{1-1/m}
\mathbb{E}_{X_1,\ldots, X_{n-1}} \bigl(c_2(m)^2|f_{j^0_n}(X_1,\ldots, X_{n-1})|^2 + \varepsilon^2\bigr)^{-1/2m}
\\
\le
C(m)c_2(m)^{-1/m}\|\varphi\|_\infty^{1/m}\|\varphi'\|_\infty^{1-1/m}
\mathbb{E}_{X_1,\ldots, X_{n-1}} \bigl(|f_{j^0_n}(X_1,\ldots, X_{n-1})|^2 + \varepsilon^2\bigr)^{-1/2m}.
\end{multline*}
Moreover, one has
\begin{multline*}
\mathbb{E}_{X_1,\ldots, X_{n-1}} \bigl(|f_{j^0_n}(X_1,\ldots, X_{n-1})|^2 + \varepsilon^2\bigr)^{-1/2m}
\\
=
\int_{0}^{\varepsilon^{-1/m}}
P\Bigl(\bigl(|f_{j^0_n}(X_1,\ldots, X_{n-1})|^2 + \varepsilon^2\bigr)^{-1/2m}\ge \tau\Bigr)\, d\tau
\\
=
\frac{1}{m}\int_{0}^{\infty}\frac{s}{(s^2+\varepsilon^2)^{1+1/2m}}\,
P\bigl(|f_{j^0_n}(X_1,\ldots, X_{n-1})|\le s\bigr)\, ds
\\
\le
\frac{8}{m}\int_{0}^{\infty}\frac{s}{(s+\varepsilon)^{2+1/m}}\,
\sigma(\varrho_{f_{j^0_n}}, s)\, ds
=
\frac{8}{m}\varepsilon^{-1/m}\int_{0}^{\infty}\frac{t}{(t+1)^{2+1/m}}\,
\sigma(\varrho_{f_{j^0_n}}, \varepsilon t)\, dt.
\end{multline*}
We now note that $d[f_{j^0_n}]\le d-1$ and $a[f_{j^0_n}]=a[f]=1$.
Thus, by the inductive hypothesis,
one has
\begin{multline*}
\sigma(\varrho_{f_{j^0_n}}, \varepsilon t)
\le
C(m, d-1)(\varepsilon t)^{1/m}\bigl[|\ln \varepsilon t|^{d-1-m}+1\bigr]
\\
\le
2^{d-1-m}C(m, d-1)(\varepsilon t)^{1/m}\bigl[|\ln t|^{d-1-m}+|\ln \varepsilon|^{d-1-m}+1\bigr],
\end{multline*}
which implies that
\begin{multline*}
\varepsilon^{-1/m}\int_{0}^{\infty}\frac{t}{(t+1)^{2+1/m}}\,
\sigma(\varrho_{f_{j^0_n}}, \varepsilon t)\, dt
\le
\varepsilon^{-1/m}\int_{0}^{\infty}\frac{1}{(t+1)^{1+1/m}}\,
\sigma(\varrho_{f_{j^0_n}}, \varepsilon t)\, dt
\\
\le
2^{d-1-m}C(m, d-1)\int_{0}^{\varepsilon^{-1}}
\frac{t^{1/m}\bigl[|\ln t|^{d-1-m}+|\ln \varepsilon|^{d-1-m}+1\bigr]}{(t+1)^{1+1/m}}\, dt
+
\varepsilon^{-1/m}\int_{\varepsilon^{-1}}^{\infty}\frac{1}{(t+1)^{1+1/m}}\, dt
\\
\le
2^{d-1-m}C(m, d-1)\int_{0}^{\varepsilon^{-1}}
\frac{\bigl[|\ln t|^{d-1-m}+|\ln \varepsilon|^{d-1-m}+1\bigr]}{t+1}\, dt
+
(1+\varepsilon)^{-1/m}.
\end{multline*}
We now assume that $\varepsilon\in(0,e^{-1}]$.
In this case
\begin{multline*}
\int_{0}^{\varepsilon^{-1}}
\frac{\bigl[|\ln t|^{d-1-m}+|\ln \varepsilon|^{d-1-m}+1\bigr]}{t+1}\, dt
\\
=
\int_{0}^{1}
\frac{|\ln t|^{d-1-m}}{t+1}\, dt
+
\int_{1}^{\varepsilon^{-1}}
\frac{|\ln t|^{d-1-m}}{t+1}\, dt
+
\bigl[|\ln \varepsilon|^{d-1-m}+1\bigr]\ln(1+\varepsilon^{-1})
\\
\le
c(d,m) + |\ln\varepsilon|^{d-1-m}\ln (1+\varepsilon^{-1})
+ \bigl[|\ln \varepsilon|^{d-1-m}+1\bigr]\ln(1+\varepsilon^{-1})
\le 10|\ln\varepsilon|^{d-m} + c(d,m).
\end{multline*}
Thus, for $\varepsilon\in(0,e^{-1}]$ we have
$$
\mathbb{E}\, \varphi'\bigl(f(X) + \varepsilon Z\bigr)
\le
C_1(m,d)\|\varphi\|_\infty^{1/m}\|\varphi'\|_\infty^{1-1/m}[|\ln\varepsilon|^{d-m} + 1]
$$
and
$$
\mathbb{E}\, \varphi'\bigl(f(X)\bigr)
\le
2\|\varphi'\|_\infty\, \sigma(\gamma_n\circ f^{-1}, \varepsilon)
+
C_1(m,d)\|\varphi\|_\infty^{1/m}\|\varphi'\|_\infty^{1-1/m}[|\ln\varepsilon|^{d-m} + 1],
$$
which implies that
$$
\sigma(\varrho_f, t)
\le
2 \sigma(\varrho_f, \varepsilon)
+
C_1(m,d)\, t^{1/m}[|\ln\varepsilon|^{d-m} + 1]
$$
for any $t>0$ and $\varepsilon\in(0,e^{-1}]$.
We first consider the case $t\in(0,e^{-1}]$.
In this case,
$$
\sigma(\varrho_f, t)
=
\sum\limits_{k=0}^{\infty}2^k\Bigl[\sigma\bigl(\varrho_f, t2^{-2dk}\bigr) -
2\sigma\bigl(\varrho_f, t2^{-2d(k+1)}\bigr)\Bigr],
$$
since, by Theorem \ref{t1}, one has
$$
2^{k+1}\sigma\bigl(\varrho_f, t2^{-2d(k+1)}\bigr)
\le
C(d)[\mathbb{D}f(X)]^{-1/2d}\, t^{1/d}2^{-k-1}\to0.
$$
Hence
$$
\sigma(\varrho_f, t)
\le
C_1(m,d)t^{1/m}\sum\limits_{k=0}^{\infty}2^{k(1-2d/m)}[|\ln t2^{-2d(k+1)}|^{d-m} + 1].
$$
Note that
\begin{multline*}
\sum\limits_{k=0}^{\infty}2^{k(1-2d/m)}[|\ln t2^{-2d(k+1)}|^{d-m} + 1]
\\
\le
2^{d-m}[|\ln t|^{d-m}+1]\sum\limits_{k=0}^{\infty}2^{k(1-2d/m)}
+
4^{d-m}d^{d-m}\sum\limits_{k=0}^{\infty}(k+1)^{d-m}2^{k(1-2d/m)}.
\end{multline*}
Since $1-2d/m\le -d/m\le-1$, both series above converge.
Therefore,
$$
\sigma(\varrho_f, t)
\le
C(m,d)t^{1/m}[|\ln t|^{d-m}+1]
$$
for $t\in(0,e^{-1}]$.

For $t\ge e^{-1}$ we have
$$
\sigma(\varrho_f, t)\le 1\le  e^{1/m} t^{1/m}\bigl[|\ln t|^{d-m} + 1\bigr].
$$
The theorem is proved.

\section{Applications}

In this section we discuss two applications of the obtained result.

Firstly, we apply Theorem \ref{T-main}
to obtain bounds for characteristic functions.

\begin{corollary}
Let $m,d\in\mathbb{N}$, $d\ge m$,
let $X:=(X_1,\ldots, X_n)$ be the standard Gaussian $n$-dimensional
random vector with independent coordinates.
Then there is a constant $C(m, d)$, depending only on $d$ and $m$,
such that,
for any non-constant polynomial
$$
f(x):=\sum_{j_1=0}^{m}\ldots\sum_{j_n=0}^{m}a_{j_1,\ldots, j_m}x_1^{j_1}\cdot\ldots\cdot x_n^{j_n}
$$
of degree $d[f]:=\max\{j_1+\ldots+j_n\colon a_{j_1,\ldots, j_m}\ne0\}\le d$,
one has
$$
|\mathbb{E}\exp\{it f(X)\}| \le C(n,m)|a[f]\cdot t|^{-1/m}\bigl[\bigl|\ln|a[f]\cdot t|\bigr|^{d-m}+1\bigr]
$$
where $a[f]:=\max\limits_{j_1+\ldots+j_n=d[f]}|a_{j_1,\ldots, j_n}|$.
\end{corollary}

\begin{proof}
As we have already proved,
one has
$$
\sigma(\varrho_f, \varepsilon)\le
C(m,d) (\varepsilon/a[f])^{1/m}\bigl[|\ln(\varepsilon/a[f])|^{d-m} + 1\bigr]
$$
which means that for any function $\varphi\in C_0^\infty(\mathbb{R})$
with $\|\varphi'\|_\infty\le1$ and $\|\varphi\|_\infty\le |t|^{-1}$
one has
$$
\mathbb{E}\varphi'\bigl(f(X)\bigr)
\le
C(m,d) |a[f]\cdot t|^{-1/m}\bigl[\bigl|\ln|a[f]\cdot t|\bigr|^{d-m} + 1\bigr].
$$
We now take $\varphi(s)= \pm t^{-1}\cos(ts)$ and $\varphi(s)=\pm t^{-1}\sin(ts)$
and get the announced bound.
\end{proof}

Secondly, we use Theorem \ref{T-main}
to obtain bounds between the total variation
and the Kantorovich--Rubinstein distances.

Let $X,Y$ be two random variables.
The total variation distance
is defined by the equality
$$
d_{\rm TV}(X,Y)  := \sup\biggl\{\mathbb{E}\bigl[\varphi(X) - \varphi(Y)\bigr]\colon \varphi\in C_0^\infty(\mathbb{R}),
\ \|\varphi\|_\infty \le1 \biggr\}.
$$
The Kantorovich--Rubinstein distance is defined by the formula
$$
d_{\rm KR}(X,Y) := \sup\biggl\{\mathbb{E}\bigl[\varphi(X)-\varphi(Y)\bigr]\colon
\varphi\in C_0^\infty(\mathbb{R}),\ \|\varphi\|_\infty \le 1,\ \|\varphi'\|_\infty\le1\biggr\}.
$$
We recall (see \cite{B18})  that
convergence in Kantorovich--Rubinstein distance is equivalent to
 convergence in distribution (weak convergence of distributions).

In \cite[Lemma 3.1]{Kos-FCAA}, the following bound is proved.

\begin{proposition}\label{prop3}
Let $X$ and $Y$ be random variables.
Then for any $\varepsilon\in(0,1)$ one has
$$
d_{\rm TV}(X,Y) \le 6\max\{\sigma(\varrho_X, \varepsilon),\sigma(\varrho_Y, \varepsilon)\}
+ \varepsilon^{-1}d_{\rm KR}(X,Y)
$$
where $\varrho_X$ and $\varrho_Y$ are distribution densities of $X$ and $Y$, respectively.
\end{proposition}

\begin{corollary}
Let $d,m\in \mathbb{N}$ and let $a\in \mathbb{R}$ be a positive number.
Let $f$ and $g$ be two polynomials of the form (\ref{rv})
and let $X:=(X_1,\ldots, X_n)$ be the standard Gaussian $n$-dimensional
random vector with independent coordinates.
Assume that $d[f]\le d$, $d[g]\le d$, $a[f]\ge a$, and $a[g]\ge a$.
Then
$$
d_{\rm TV}(f(X), g(X))\le C(m, d, a)[d_{\rm KR}(f(X), g(X))]^{\frac{1}{m+1}}
\bigl[|\ln d_{\rm KR}(f(X), g(X))|^{\frac{(d-m)m}{m+1}}+1\bigr].
$$
\end{corollary}

\begin{proof}
We have
$$
\sigma(\varrho_f, \varepsilon)\le
C(m,d) (\varepsilon/a[f])^{1/m}\bigl[|\ln(\varepsilon/a[f])|^{d-m} + 1\bigr]
\le
C_1(m, d, a)  \varepsilon^{1/m}\bigl[|\ln\varepsilon|^{d-m} + 1\bigr]
$$
and the same bound is true for $\sigma(\varrho_g,\cdot)$.
Thus, by Proposition \ref{prop3}, one has
$$
d_{\rm TV}(f(X), g(X))
\le
6C_1(m,d,a)\varepsilon^{1/m}\bigl[|\ln\varepsilon|^{d-m} + 1\bigr]
+ \varepsilon^{-1}d_{\rm KR}(f(X),g(X))
$$
for any $\varepsilon\in(0,1)$.
Since $d_{\rm KR}(f(X), g(X))\le 2$,
we can take
$$
\varepsilon =
\bigl[\tfrac{1}{3}d_{\rm KR}(f(X), g(X))\bigr]^{\frac{m}{m+1}}
\bigl|\ln\bigl(\tfrac{1}{3}d_{\rm KR}(f(X), g(X))\bigr)\bigr|^{\frac{(m-d)m}{m+1}},
$$
which implies the announced bound.
\end{proof}

\section*{Acknowledgment}

The author is a Young Russian Mathematics award winner and would like to thank its sponsors and jury.

This reaserach is supported by the Russian Science Foundation Grant 17-11-01058 (at Lomo\-no\-sov Moscow State University).

\end{document}